\documentclass[a4paper,reqno,12pt]{amsart} 
\usepackage[latin1]{inputenc}
\usepackage[T1]{fontenc}
\usepackage[english]{babel}
\usepackage{appendix}
\usepackage[english]{minitoc}
\usepackage[nice]{nicefrac}

\usepackage[top=2cm, bottom=2cm, left=2cm, right=2cm]{geometry}
\numberwithin{equation}{section}
\makeatletter\@addtoreset{equation}{section}
\usepackage{mathtools}

\DeclarePairedDelimiter\floor{\lfloor}{\rfloor}
\usepackage{amsmath}
\usepackage{amssymb}
\usepackage{amsbsy} 
\usepackage{latexsym, amsfonts}
\usepackage{graphics}
\usepackage{ulem}
\usepackage{hhline}
\usepackage{dsfont}
\usepackage{mathrsfs}
\usepackage{color}
\usepackage{fancyhdr}
\usepackage{rotating}
\usepackage{fancybox}
\usepackage{colortbl}
\usepackage{pifont}
\usepackage{setspace}
\usepackage{enumerate}
\usepackage{multicol}
\usepackage{varioref}
\usepackage{lmodern}
\usepackage{textcomp}
\usepackage{euscript}
\usepackage[pdftex]{hyperref}

\newcommand{\bz}{\bar z}

\newcommand{\R}{\mathbb{R}}

\newcommand{\C}{\mathbb{C}}

\newtheorem {theorem}{Theorem}[section]            
   \newtheorem {corollary}[theorem]{Corollary}     \newtheorem {remark}[theorem]{Remark}
\newtheorem {proposition}[theorem]{Proposition}       
         
\begin{document}

	\title[$(p,q)$- Gould--Hopper polynomials]{Two-dimensional $(p,q)$-heat polynomials of Gould--Hopper type 
	 }
		
	\author[A. Ghanmi]{Allal Ghanmi}
	\author[K. Lamsaf]{Khalil Lamsaf}
	\address{ Analysis, P.D.E. $\&$ S.G. - Lab. M.I.A.-S.I., CeReMar,
		Department of Mathematics, P.O. Box 1014,  Faculty of Sciences,
		Mohammed V University in Rabat, Morocco}

	\email{allal.ghanmi@um5.ac.ma}
	\email{lamsafkhalil@gmail.com}
	
	\begin{abstract}
We introduce a new class of holomorphic polynomials extending the classical Gould--Hopper to two complex variables. The considered polynomials include the $1$-D and $2$-D holomorphic and polyanalytic It\^o--Hermite  polynomials as particular cases.
We emphasize to study their operational
  representation, various generating functions and recurrence relations.
We also establish some special identities including multiplication and addition formulas of Runge type,  as well as the Nielson type formulas. Higher order partial differential equations are  analyzed and the connection to Gould-Hopper polynomials and hypergeometric functions are  investigated.
	\end{abstract}
	
	\maketitle
	
	
	
\section{Introduction}

Different generalizations of the classical Hermite polynomials $H_n(x)$
to multivariate setting have been widely studied in the literature \cite{Hermite1864-1908,Ismail13b,ismail2015complex,nielsen1918recherches,Szego75,mourad2005classical} and have interesting applications in many branches of mathematics, physics and engineering.
In fact, the tensor product  $H_m(x)H_n(y)$ as well as the holomorphic Hermite polynomials  $H_{n}(z)$
are specific $2D$ generalization of $H_n(x)$ to the complex plane $\C$.
The first class is an orthogonal basis of $L^2(\R^2;e^{-x^2-y^2}dxdy)$, while the functions $e^{-z^2/2}H_m(z)$ form an orthogonal basis of a  Bargmann-Fock like space  \cite{van1990new,BenGS2019}.
 For their analytic and combinatoric properties one refer to  \cite{Szego75,Ismail13b}.
The polyanalytic analog are the It\^o--Hermite (or complex Hermite) polynomials defined by
\begin{align}
H_{n, m} (z, \bz):= (-1)^{m+n} e^{ |z|^2} \partial_{\bz}^n \partial_z^m \left( e^{- |z|^2}\right)= n!m! \sum_{j=0}^{n \wedge m}\dfrac{(-1)^{j}}{j!} \dfrac{z^{n-j}\bz^{m-j}}{(n-j)!(m-j)!} ,
\end{align}
where $n \wedge m=min\{n,m\}$, for
being solutions of the iterated  Cauchy--Riemann equation $2^{n+1}\partial_{\bz}^{n+1} f=(\partial_x + i\partial_y)^{n+1} f=0$.
These polynomials, introduced by It\^o \cite {Ito52} in the context of the complex Markov process,
are basic tools in the nonlinear analysis of traveling wave tube amplifiers \cite{Barrett1984}. More specifically, they appear in the calculation of the effects of non-linearity on broadband radio frequencies in communication systems. For their properties and applications one can refer to \cite{Gh08jmaa,Gh13ITSF,Ismail13b,DunklXu14,Gh17MMAS}.
Their holomorphic counterparts  $H_{n,m}(z, w)$ and $H_{n, m}(z, w,x)$  were introduced and studied recently in
  \cite{Gorska2019holomorphic,Ismail13b} and \cite{Z-GLiu2017}, respectively.
Another class of generalized Hermite polynomials is that we refer to as Gould--Hopper \cite[Eq. (6.2), p. 58]{GouldHopper} defined by
 \begin{eqnarray}\label{gh0}
 H_{n}^{(p)}(z|\gamma)=n!\displaystyle\sum_{k=0}^{\floor*{\frac{n}{p}}}\frac{\gamma^k}{k!}\frac{z^{n-pk}}{(n-pk)!}.
 \end{eqnarray}
They enter in the study of Novikov--Veselov equation \cite{chang2011gould}. The specific ones  $H_{n}^{(2p)}(x|(-1)^{p+1}t) $ are solutions of the higher-order heat equation associated to $(-1)^{q+1} {\partial^{2 q}}/{\partial x^{2 q}}$ (see e.g. \cite{haimo1992representation,GouldHopper,leveque2011sum}).
%

In the present paper, we introduce and study the basic properties of the
 natural extension of $H_{n}^{(p)}(z|\gamma)$ in \eqref{gh0} to two complex variables
 \begin{equation}\label{def1}
 	H_{n,m}^{(p,q)}(z,w|\gamma)= n!m!\displaystyle\sum_{k=0}^{\floor*{\frac{n}{p}}\wedge\floor*{\frac{m}{q}}}\frac{\gamma^k}{k!}\frac{z^{n-pk}}{(n-pk)!}\frac{w^{m-qk}}{(m-qk)!}
 \end{equation}
 which contains all the classes mentioned earlier and gives rise to new ones.
 More precisely, we have
 \begin{enumerate}
 	\item $H_{n,0}^{(2,0)}(2z,1|-1)=H_{n}(z)$, the holomorphic Hermite  polynomials. 
 	\item $H_{n,0}^{(p,0)}(z,1|\gamma)= H_{n}^{(p)}(z|\gamma)$,  the Gould-Hopper polynomials.
 	We also have $ H_{n,m}^{(p,0)}(z,w,|\gamma) = H_{n,m}^{(0,p)}(z,w,|\gamma)=w^{m}H_{n}^{(p)}(z|\gamma).$
 	\item $
 	H_{n,m}^{(1,1)}(z,w|-1)=H_{n,m}(z,w)$ and $ H_{n,m}^{(1,1)}(z,\overline{z}|-1)=H_{n,m}(z,\overline{z}) $, the $2$-D holomorphic Hermite polynomials and their restriction to the non-analytic surface $w=\bz$, respectively.
 	
 	\item $H_{n,m}^{(1,1)}(z,\overline{z}|\gamma)=H_{n,m}(z,w,\gamma)$, the  Hermite polynomials considered in \cite{Z-GLiu2017}.
 \end{enumerate}
 {Meaningful and significant application of the considered polynomials is
 	the representation solutions of the partial differential equation in the $(z,w)$-plane
 	\begin{equation}\label{hoheateq2}
 		c_{p,q} \frac{\partial^{p+q}}{\partial z^{ p}\partial w^{q}} u(z,w;t)=\frac{\partial}{\partial t} u(z,w; t) .
 	\end{equation}
 	with initial data of analytic.}
 We call them here two-dimensional complex $(p,q)$-heat polynomials of Gould--Hopper type.
 {Other motivations of considering $H_{n,m}^{(p,q)}(x,y|\gamma)$ are eventual applications in quantum mechanics, combinatorics or applied mathematics. Mainly, they can be used
 	in evaluating transition matrix elements, studying the root dynamics of the  $\sigma$-flows  associated  to
 	\eqref{hoheateq2}
 	in a similar way as done in \cite{TaimanovTsarev2008} or also in computing  the higher-order moments of a given distribution.
 }

 Even if this extension is natural and  some of the algebraic properties seem to be derived in a similar way,  there is no evidence to suggest
 exact formulas or examine their analytical properties including  their orthogonality, description of associated functional spaces and integral transforms.
Namely, we  provide  a complete unified description of basic properties of $H_{n,m}^{(p,q)}(x,y|\gamma)$.
More precisely, we are concerned with  their
	operational and hypergeometric  representations (Section 2),
generating functions (Section 3),
addition formulas of Runge type (Section 4),
multiplication formulas (Section 5),
recurrence relations (Section 6),
Nielson type formulas (Section 7),
and the connection to Gould--Hopper polynomials (Section 8).
We also discuss some
higher order partial differential equations (Section 9).



\section{Operational and hypergeometric representations}

	The polynomials $H_{n,m}^{(p,q)}(z,w|\gamma)$ we deal with are the two-dimensional complex of Gould--Hopper type given through \eqref{def1}  with the convention that $H_{0,0}^{(0,0)}(z,w|\gamma)=e^{\gamma}$ and $\floor*{\frac{j}{k}}=+\infty$ when $k=0$.
		 Obviously, they generalize the real, Gould--Hopper, as well as the $1$-D and $2$-D holomorphic, and Ito--Hermite polyanalytic complex polynomials.
	The monomials $z^nw^m$ can be recovered by taking $\gamma=0$ or $n<p$ and $m<q$,  $H_{n,m}^{(p,q)}(z,w|\gamma)=z^{n}w^{m}$, while for $n = p$ and $m = q$, we have  $H_{p,q}^{(p,q)}(z,w|0)=z^{p}w^{q}+{\gamma}p!q!.$
		 %
	Notice also that the value at $z=w=0$ gives rise to
		 $$H_{n,m}^{(p,q)}(0,0|\gamma)
		 =\frac{n!}{\floor*{\frac{n}{p}}!} \gamma^{\floor*{\frac{n}{p}}} \delta_{\floor*{\frac{n}{p}},\floor*{\frac{m}{q}}}$$
		when $p \mid  n$ and $q\mid m$.
		 Otherwise, we have  
		 $H_{n,m}^{(p,q)}(0,w|\gamma)=0$ if $p \nmid  n$ and $ H_{n,m}^{(p,q)}(z,0|\gamma)=0$ if $q \nmid  m$.
		 The case $p=n=0$, gives rise to
		 $H_{0,qm}^{(0,q)}(0,0|\gamma) =
		 {(qm)!}  \gamma^{m}/{m!} $ and $H_{0,qm+j}^{(0,q)}(0,0|\gamma) = 0$ when $j = 1,\dots,q-1$
		 so that one recoups the well-known identity for the real Hermite polynomials
		 $
		 H_{2n}(0)=(-1)^{n} {(2 n) !}/{n !}$ and $H_{2n+1}(0)=0$. 

	The following operational formula can be considered as an equivalent definition
	which offers a broader direction, by generalizing $e^{\partial_z}$ to several variables and to higher order in the exponent.

	\begin{proposition}\label{prop:Realiz}
		We have
		\begin{equation}\label{opFor}
		H_{n,m}^{(p,q)}(z,w|\gamma)=e^{\gamma\partial_{z}^{p}\partial_{w}^{q}}\{z^{n}w^{m}\}.
		\end{equation}
		Moreover, we have the symmetry property
		\begin{equation}\label{sym}
		H_{n,m}^{(p,q)}(z,w|\gamma) =H_{m,n}^{(q,p)}(w,z|\gamma) .
		\end{equation}
	\end{proposition}
	
This follows making use of the fact that
	$\partial_{z}^{pk}\{z^{n}\}=
	\frac{n!z^{n-pk}}{(n-pk)!} $
	for $k\leq\floor*{\frac{n}{p}}$ and vanishes otherwise. 
	
The next result gives the hypergeometric representation of   $H_{n,m}^{(p,q)}(z,w|\gamma)$ in terms of the ${_rF_s}$ hypergeometric function
$$
{_rF_s}\left(
\begin{array}{c}
a_1,a_2,\cdots,a_r\\ c_1,c_2,\cdots,c_s
\end{array}
\bigg| x\right) := \sum_{k=0}^\infty \frac{(a_1)_k (a_2)_k\cdots (a_r)_k}{(c_1)_k (c_2)_k \cdots (c_s)_k}\frac{x^k}{k!} .
$$

\begin{proposition}\label{thmHypg}
	We have
	\begin{equation}	\label{hyperg}
	H_{n,m}^{(p,q)}(z,w|\gamma)=z^{n}w^{m}{_{p+q}F_0}\left(
	\begin{array}{c}
	\frac{-n}{p},\frac{1-n}{p},\cdots,\frac{p-1-n}{p},\frac{-m}{q},\frac{1-m}{q},\cdots,\frac{q-1-m}{q}
	\\ \_
	\end{array}
	\bigg| \frac{(-p)^{p}(-q)^{q}\gamma}{z^{p}w^{q}} \right) .
	\end{equation}
\end{proposition}

\begin{proof}
	Starting from \eqref{def1}, we can rewrite the explicit expression of  $H_{n,m}^{(p,q)}$ in the following form
	$$H_{n,m}^{(p,q)}(z,w|\gamma)= n!m!z^{n}w^{m}\displaystyle\sum_{k=0}^{\floor*{\frac{n}{p}}\wedge\floor*{\frac{m}{q}}}
	\frac{\left({\gamma}/{z^{p}w^{q}}\right)^k}{k!(n-pk)!(m-qk)!} 	
	.$$
Appealing the identity  \cite[Eq. (21), p. 21]{manocha1984treatise} as good as the Gauss multiplication theorem \cite[Eq. (26), p. 23]{manocha1984treatise},
 we can rewrite the involved factorial $(n-pk)!$  as
		$$
	(n-pk)!=\frac{(-1)^{pk}n!}{(-p)^{pk} \prod_{j=1}^{p}\left(\frac{j-1-n}{p}\right)_{k}}.
	$$
Therefore, we obtain
	$$H_{n,m}^{(p,q)}(z,w|\gamma)= z^{n}w^{m}\displaystyle\sum_{k=0}^{\floor*{\frac{n}{p}}\wedge\floor*{\frac{m}{q}}} \prod_{j=1}^{p}\left(\frac{j-1-n}{p}\right)_{k}
		\prod_{j=1}^{q}\left(\frac{j-1-m}{q}\right)_{k}
	\frac{\left(\frac{(-1)^{p+q}p^{p}q^{q}\gamma}{z^{p}w^{q}}\right)^k}{k!}.$$
	This is exactly the desired result \eqref{hyperg}.
\end{proof}


Subsequently, we can express the classical classes of Hermite polynomials described in the introductory section in terms of the confluent hypergeometric function ${_1F_1}$.
By means of the following transformation formula
	\begin{align}\label{lemma-H.F}
	{}_{2}F_{0}\left(\begin{array}{c}-n,-m\\ \_\end{array} \bigg| \frac{-1}{z}\right)=\frac{z^{-(n\wedge m)}(n\vee m)!}{(|n-m|)!}{}_{1}F_{1}\left(\begin{array}{c}-(n\wedge m)\\ |n-m|+1\end{array} \bigg| z\right).
	\end{align}
where $n\vee m= max(n, m)$, which can be handled by applying 
\begin{align}
	{}_{1}F_{1}\left(\begin{array}{c}a\\ b\end{array} \bigg| z\right)&=\frac{\Gamma(b)}{\Gamma(b-a)}(-z)^{-a}{}_{2}F_{0}\left(\begin{array}{c}a,1+a-b\\ \_\end{array} \bigg| \frac{-1}{z}\right)
	\\& \qquad +\frac{\Gamma(b)}{\Gamma(a)} z^{a-c}e^{z}{}_{2}F_{0}\left(\begin{array}{c} b-a,1-a\\ \_ \end{array} \bigg| \frac{1}{z}\right) \nonumber
\end{align}
with $a=-n,b=m-n+1$ when $n\leq m$ and by next by the symmetry property to get the case $n\geq m$.
  Thus, we retrieve the ${}_{1} F_{1}$ representations  \cite[Eq. (1.2)]{paris2010asymptotics} for  $H_{n}(z)$ and  \cite[Eq. (2.3)]{intissar2006spectral} for $H_{n,m}^{(1,1)}(z,\bz)$.

\section{Generating functions.}
In this section, we derive generating functions for the polynomials $H_{n,m}^{(p,q)}$. 

\begin{proposition}
	We have
	\begin{align}
	& \sum_{n=0}^\infty  H_{n,m}^{(p,q)}(z,w|\gamma)\frac{u^{n}}{n!}=H_{m}^{(q)}(w|u^p\gamma)e^{zu} \label{fg1},
	\\&
	\sum_{m=0}^\infty  H_{n,m}^{(p,q)}(z,w|\gamma)\frac{v^{m}}{m!}=H_{n}^{(p)}(z|v^q\gamma)e^{wv} \label{fg2}
	\\&
 \sum_{n=0}^\infty\sum_{m=0}^\infty  H_{n,m}^{(p,q)}(z,w|\gamma)\frac{u^{n}v^{m}}{n!m!}=e^{zu+wv+{\gamma}u^{p}v^{q}}. \label{genRfg3}
	\end{align}
\end{proposition}

\begin{proof}
	We only need to establish \eqref{fg1} since the second one is its symmetry. 
	By \eqref{opFor} we have
	\begin{align*}
	\sum_{n=0}^\infty  H_{n,m}^{(p,q)}(z,w|\gamma)\frac{u^{n}}{n!} 
	=e^{\gamma\partial_{z}^{p}\partial_{w}^{q}}\{e^{zu}w^m\}
	=e^{\gamma u^{p}\partial_{w}^{q}}\{w^{m}\}e^{zu}
	=H_{m}^{(q)}(w|u^p\gamma)e^{zu}.
	\end{align*}
	Now, let denote the left hand-side of \eqref{genRfg3} by	$R^{p,q}_\gamma(z,w|u,v)$. Hence, by means of \eqref{fg1}, it follows that
	\begin{align*}
	R^{p,q}_\gamma(z,w|u,v)	
	&= \sum_{m=0}^\infty  \left(\sum_{n=0}^\infty  H_{n,m}^{(p,q)}(z,w|\gamma)\frac{u^{n}}{n!}\right) \frac{v^{m}}{m!}
	=\sum_{m=0}^\infty H_{m}^{(q)}(w|u^{p}\gamma)\frac{v^{m}}{m!}e^{zu}.
	\end{align*}
Then, the result in \eqref{genRfg3} follows using the generating function of the Gould--Hopper polynomials \cite[p.72]{dattoli1994theory}
	\begin{align}\label{GenFctGH}
	\sum_{m=0}^\infty  H_{m}^{(q)}(w|\gamma) \frac{v^{m}}{m!}=e^{wv+\gamma{v}^{q}}.
	\end{align}
\end{proof}

\begin{remark}
	For $p=q= 1$, we retrieve from \eqref{fg2} the partial generating functions for the complex
	Hermite polynomials given through \cite[Proposition 3.4]{Gh13ITSF},
while \eqref{genRfg3} generalizes the corresponding formula for the Hermite polynomials described in the introductory section.
\end{remark}



%

As immediate consequence of the generating function given through \eqref{genRfg3} is the following identity needed to prove some of our later results.

\begin{corollary}
	 We have
	\begin{equation}\label{homogeneity}
	a^{n}b^{m}H_{n,m}^{(p,q)}(z,w|\gamma)=H_{n,m}^{(p,q)}(az,b w|{\gamma} a^{p} b^{q}).
	\end{equation}
\end{corollary}

\begin{proof}
	Equation \eqref{homogeneity} follows from the observation
	$e^{a u z + b v w +{\gamma}(au)^{p}(bv)^{q}}
	=e^{a u z + b v w +({\gamma}a^{p}b^{q})u^{p}v^{q}}$
	which combined with \eqref{genRfg3}  yields
		\begin{align*}
	\sum_{n=0}^\infty \sum_{m=0}^\infty a^{n}b^{m}H_{n,m}^{(p,q)}(z,w|\gamma)\frac{u^n}{n!}\frac{v^m}{m!}
	=	\sum_{n=0}^\infty \sum_{m=0}^\infty H_{n,m}^{(p,q)}(az,bw|{\gamma}a^{p}b^{q})\frac{u^n}{n!}\frac{v^m}{m!}.
	\end{align*}
\end{proof}


The next consequence is an extension to the $(p,q)$ complex Hermite polynomials of the limit identities \cite{van1990new,Gorska2019holomorphic}
$$
\lim _{t \rightarrow 0}\left(\frac{t}{2}\right)^{n} H_{n}\left(\frac{z}{t}\right)=z^{n}
\quad \mbox{ and } \quad
\lim _{t \rightarrow 0} t^{m+n} H_{m, n}\left(\frac{z}{t}, \frac{w}{t}\right)=z^{m} w^{n}
$$
for the holomorphic Hermite polynomials $H_{n}$ and $H_{n,m}$.

\begin{corollary}
	For $p + q \geq 0$, we have
	\begin{equation}\label{limit}
	\lim_{t\to 0}t^{n+m}H_{n,m}^{(p,q)}\left( \frac{z}{t},\frac{w}{t}\Big|\gamma\right) =z^n w^m.
	\end{equation}
	\end{corollary}

\begin{proof}
	By means of \eqref{homogeneity} with
	$a = b = t$ and $z := \frac{z}{t}, w := \frac{w}{t}$, and tending $t$ to zero, we get
$$\lim_{t\to 0}t^{n+m}H_{n,m}^{(p,q)}\left( \frac{z}{t},\frac{w}{t} \Big| \gamma\right) = \lim_{t\to 0} H_{n,m}^{(p,q)}(z,w|t^{p+q}\gamma) =H_{n,m}^{(p,q)}(z,w|0)=
=0.$$
\end{proof}


We conclude this section by establishing he closed expression of the generating functions
$$ S^{p,q,\gamma}_{a,b}(z,w|u,v) = \sum_{n=0}^{\infty}\sum_{m=0}^{\infty}(a)_{n}(b)_{m}H_{n,m}^{(p,q)}(z,w|\gamma)\frac{u^{n}v^{m}}{n!m!}$$
and
\begin{equation}\label{Gf1}
G ^{p,q,\gamma}_{j,k}(z,w|u,v) : = 	\sum_{n=0}^\infty\sum_{m=0}^\infty  (n)_{j}(m)_{k}H_{n,m}^{(p,q)}(z,w|\gamma)\frac{u^{n}v^{m}}{n!m!}.
\end{equation}
We set
$$
P_{k}^{n}(z)=\sum_{j=0}^{k-1}(-1)^{(k-j)}(-n)_{(k-j)}\binom{k}{j} z^{j}.
$$

	\begin{theorem}
		 We have
		\begin{equation} \label{gfg}
	G ^{p,q,\gamma}_{j,k}(z,w|u,v)	=uvzwe^{zu+wv+{\gamma}u^{p}v^{q}}\left((uz)^{j-1}+P_{j-1}^{j}(uz)\right)
	\left((vw)^{k-1}+P_{k-1}^{k}(vw)\right)
		\end{equation}
		and
			\begin{align}\label{SGenHyperg}
		S^{p,q,\gamma}_{a,b}(z,w|u,v)	&=(1-uz)^{-a}(1-vw)^{-b}
		\\&\times
		{}_{p+q}F_{0} \left(
		\begin{array}{c}
		\frac{a}{p},\frac{a+1}{p},\cdots,\frac{a+p-1}{p},\frac{b}{q},\frac{b+1}{q},\cdots,\frac{b+q-1}{q}
		\\ \_
		\end{array}
		\bigg| \frac{p^{p}q^{q}{\gamma}uv}{(1-uz)^{p}(1-vw)^{q}} \right). \nonumber
		\end{align}
	\end{theorem}

	\begin{proof}
Using the definition \eqref{def1} of $H_{n,m}^{(p,q)}$, we can rewrite $G ^{p,q,\gamma}_{j,k}(z,w|u,v)$ as
\begin{align*}
G ^{p,q,\gamma}_{j,k}(z,w|u,v)
&=\sum_{k=0}^{\infty}\frac{(\gamma u^{p}v^{q})^k}{k!}\left( \sum_{n=0}^{\infty}(n)_{j}\frac{(uz)^n}{n!}
\sum_{m=0}^\infty(m)_{k}\frac{(vw)^m}{m!}\right) \\&=e^{{\gamma}u^{p}v^{q}}uz e^{uz}\left((uz)^{j-1}+P_{j-1}^{j}(uz)\right)vw e^{vw}\left((vw)^{k-1}+P_{k-1}^{k}(vw)\right)
\\&=uvzwe^{zu+wv+{\gamma}u^{p}v^{q}}\left((uz)^{j-1}+P_{j-1}^{j}(uz)\right)
\left((vw)^{k-1}+P_{k-1}^{k}(vw)\right).
\end{align*}
The second equality results from the use of the generating function \cite[Theorem 2.1]{Petojevic2018},
 $$  \sum_{n=0}^{\infty}(n)_{j} \frac{z^{n}}{n !}=z e^{z}\left(z^{j-1}+P_{j-1}^{j}(z)\right).$$

	Now starting from \eqref{def1} and using making the identity \cite[Eq. (5), p. 101]{manocha1984treatise}
	combined with the formula \cite[Eq. (20), p. 22]{manocha1984treatise}, we obtain
	\begin{align*}
	S^{p,q,\gamma}_{a,b}(z,w|u,v) 
	&=\sum_{n=0}^{\infty}\sum_{m=0}^{\infty}\sum_{k=0}^{\infty}
	\frac{\gamma^{k}}{k!}(a)_{n+pk}(b)_{m+qk}u^{pk}v^{qk}\frac{(uz)^n}{n!}\frac{(vw)^m}{m!}
	\\&=\sum_{n=0}^{\infty}\sum_{m=0}^{\infty}\sum_{k=0}^{\infty}
	\frac{\gamma^{k}}{k!}(a)_{pk}(a+pk)_{n}(b)_{qk}(b+qk)_{m}u^{pk}v^{qk}\frac{(uz)^n}{n!}\frac{(vw)^m}{m!}.
	\end{align*}
	Applying  the Gauss multiplication theorem \cite[Eq. (26), p. 23]{manocha1984treatise} 
	  and the binomial theorem ${_1F_0}( a
	|t )=(1-t)^{-a}$,
	we get
	\begin{align*}
	S^{p,q,\gamma}_{a,b}(z,w|u,v)
	&=(1-uz)^{-a}(1-vw)^{-b}\sum_{k=0}^{\infty}
	\prod_{j=1}^{p}\left(\frac{a+j-1}{p}\right)_{k}
	\prod_{j=1}^{q}\left(\frac{b+i-1}{q}\right)_{k}
	\\& \quad \times \frac{1}{k!}\left(\frac{{p^{p}q^{q}\gamma}uv}{(1-uz)^{p}(1-vw)^{q}}\right)^{k}.
	\end{align*}
	which gives the right-hand side of \eqref{SGenHyperg}.
\end{proof}

\begin{remark}
	The identity \eqref{gfg} can be seen as a special generalization of \eqref{genRfg3}.
	Notice also that for$p = q = 1, w = z$, and $\gamma=-1$, we retrieve \cite[Eq. (4.15)]{Ismail13b}.
		 When $p = 2, m = q = 0, z = 2x$, and $\gamma=-1$, we get \cite[p. 190]{rainville1971special}
	$$
	\sum_{n=0}^{\infty} \frac{(a)_{n}}{n !} t^{n} H_{n}(x)=(1-2 x t)^{-a} {}_{2} F_{0}\left(
	\begin{array}{c}
	\frac{(a)}{2},\frac{(a+1)}{2}\\ \_
	\end{array}
	\bigg|\frac{-4 t^{2}}{(1-2 x t)^{2}}\right)
	$$
	and
	\end{remark}

\section{Runge formulas}
In this section, we establish a Runge type formula for $	H_{n,m}^{(p,q)}$. It extends the known Runge type formulas for the real Hermite  polynomials \cite{runge1914besondere}
and the It\^o--Hermite polynomials \cite{Gh13ITSF}.

\begin{proposition}\label{RungeGen}
	We have
	\begin{equation}\label{R.f2}
	H_{n,m}^{(p,q)}(z+z',w+w'|\gamma+{\gamma}')=
\sum_{k=0}^{n}\sum_{j=0}^{m}\binom{n}{k}\binom{m}{j}
	H_{k,j}^{(p,q)}(z,w|\gamma)
H_{n-k,m-j}^{(p,q)}(z',w'|{\gamma}').
	\end{equation}
\end{proposition}

\begin{proof} We consider
	$$ T^{p,q}_{\gamma,\gamma'}(z,z',w,w'|u,v) :=\sum_{n=0}^\infty\sum_{m=0}^\infty  H_{n,m}^{(p,q)}(z+z',w+w'|\gamma+{\gamma}')\frac{u^{n}v^{m}}{n!m!} .$$
	Using \eqref{genRfg3} and rewriting
	 $e^{(z+z')u+(w+w')v+({\gamma}+{\gamma}')u^{p}v^{q}}$ as $e^{zu+wv+{\gamma}u^{p}v^{q}}e^{z'u+w'v+{\gamma}'u^{p}v^{q}}$,
	 we obtain
	\begin{align*}
	T^{p,q}_{\gamma,\gamma'}(z,z',w,w'|u,v)
	&=\left(\sum_{n=0}^\infty\sum_{m=0}^\infty  H_{n,m}^{(p,q)}(z,w|\gamma)\frac{u^{n}v^{m}}{n!m!}\right)
	\left( \sum_{n'=0}^{\infty} \sum_{m'=0}^{\infty} H_{n',m'}^{(p,q)}(z',w'|{\gamma}')\frac{u^{n'}v^{m'}}{n'!m'!}\right)
	\\&=\sum_{n=0}^\infty \sum_{m=0}^\infty \left( \sum_{k=0}^{n}\sum_{j=0}^{m}
	\frac{H_{k,j}^{(p,q)}(z,w|\gamma)}{k!j!}
	\frac{H_{n-k,m-j}^{(p,q)}(z',w'|{\gamma}')}{(n-k)!(m-j)!}\right) u^{n}v^{m}.
	\end{align*}
	Equating the coefficients of $u^n v^m$
	in the last set of equations, we obtain \eqref{R.f2}.
\end{proof}

\begin{remark}
The special case  of $\gamma=-\gamma'$ in \eqref{R.f2} leads to the following identity
	\begin{equation*}\label{Identity}
\sum_{k=0}^{n}\sum_{j=0}^{m} \binom{n}{k} \binom{m}{j}
H_{k,j}^{(p,q)}\left( \frac{z}{2},\frac{w}{2} \big|\gamma\right)  H_{n-k,m-j}^{(p,q)}\left( \frac{z}{2},\frac{w}{2}\big |-\gamma\right)
= z^{n}w^m.
\end{equation*}
\end{remark}

\begin{remark}
	The special case $z=z':=\frac{z}{2}$, $w=w':=\frac{w}{2}$ and $\gamma=\gamma'=\frac{\gamma}{2}$, of \eqref{R.f2} combined with \eqref{homogeneity} gives rise to the identity
	\begin{equation*}
	H_{n,m}^{(p,q)}(z,w|\gamma)=
\sum_{k=0}^{n}\sum_{j=0}^{m}\binom{n}{k} \binom{m}{j}
	H_{k,j}^{(p,q)}(z,w|2^{p+q-1}\gamma)
	H_{n-k,m-j}^{(p,q)}(z,w|2^{p+q-1}\gamma).
	\end{equation*}
\end{remark}

As immediate consequence of Theorem \ref{RungeGen}, we assert the following.

\begin{corollary}\label{Runge1}
	 We have
	\begin{align}\label{R.f1}
	H_{n,m}^{(p,q)}\left( \frac{z+z'} {\sqrt[2p\,]{2}},\frac{w+w'}{\sqrt[2q\,]{2}} \bigg|\gamma\right) =
2^{-(\frac{n}{2p}+\frac{m}{2q})}\sum_{k=0}^{n}\sum_{j=0}^{m}\binom{n}{k}\binom{m}{j}
	H_{k,j}^{(p,q)}(z,w|\gamma)
	H_{n-k,m-j}^{(p,q)}(z',w'|\gamma).
	\end{align}
\end{corollary}

\begin{proof}
	Equation \eqref{R.f1} follows
	from the case $\gamma = \gamma'$ of Theorem \ref{RungeGen} and equation \eqref{homogeneity} with $a = 2^{{1}/{2p}}$ and $b = 2^{{1}/{2q}}$.
\end{proof}

\begin{remark}
	For $p=q=1$ and $w=\bar{z}$, we recover the Runge formula for the Ito--Hermite polynomials
	in \cite[p.9, Eq 3.22]{Gh13ITSF}
While for $ p=2$ and $m=q=0$, we recover the one for the real Hermite polynomials 	 \cite{runge1914besondere}.
\end{remark}

\section{Multiplication formulas}

We begin with the following multiplication formula needed to prove a recursion relation with respect to parameters $p$ and $q$.

\begin{proposition}\label{propm1}
	We have the identity
	\begin{equation}\label{mf1}
	H_{n,m}^{(p,q)}(z,w|c\gamma)=n!m!\sum_{k=0}^{\floor*{\frac{n}{p}}\wedge\floor*{\frac{m}{q}}}\frac{(c-1)^{k}{\gamma}^k}{k!}\frac{H_{n-pk,m-qk}^{(p,q)}(z,w|\gamma)}{(n-pk)!(m-qk)!}.
	\end{equation}
\end{proposition}

\begin{proof}
From
\begin{align*}
	\sum_{n=0}^\infty\sum_{m=0}^\infty  H_{n,m}^{(p,q)}(z,w|c\gamma)\frac{u^{n}v^{m}}{n!m!}
	&=e^{uz+vw+{\gamma}u^{p}v^{q}}e^{(c-1){\gamma}u^{p}v^{q}},
	\\&=\sum_{n=0}^\infty\sum_{m=0}^\infty \sum_{k=0}^{\infty}\frac{(c-1)^k{\gamma}^k}{k!}H_{n,m}^{(p,q)}(z,w|\gamma)\frac{u^{n+pk}v^{m+qk}}{n!m!}.
	\end{align*}
Then \eqref{propm1} follows by equating the coefficients of $u^{n}v^m$ on both sides of the last equation.
\end{proof}

Another multiplication formula is the following.

\begin{proposition} \label{propabcMu}
We have
	\begin{equation}\label{mf2}
	H_{n,m}^{(p,q)}(az,bw|c\gamma)=n!m!\sum_{k=0}^{\floor*{\frac{n}{p}}\wedge\floor*{\frac{m}{q}}} \frac{(c-a^{p}{b}^q)^{k}{\gamma}^{k}a^{n-pk}{b}^{m-qk}}{k!(n-pk)!(m-qk)!} H_{n-pk,m-qk}^{(p,q)}(z,w|\gamma).
	\end{equation}
\end{proposition}

\begin{proof}
	Applying \eqref{homogeneity}, we get
	$$H_{n,m}^{(p,q)}(az,bw|c\gamma)=a^{n}{b}^{m}H_{n,m}^{(p,q)}(z,w|ca^{-p}b^{-q}\gamma).$$
	By \eqref{mf1}, we obtain
	\begin{align*}
	H_{n,m}^{(p,q)}(az,bw|c\gamma)&=n!m!a^{n}{b}^{m}\sum_{k=0}^{\floor*{\frac{n}{p}}\wedge\floor*{\frac{m}{q}}}\frac{(ca^{-p}b^{-q}-1)^{k}{\gamma}^k}{k!}\frac{H_{n-pk,m-qk}^{(p,q)}(z,w|\gamma)}{(n-pk)!(m-qk)!}
	\end{align*}
	which reduces to the right-hand side of \eqref{mf2}.
\end{proof}

\begin{remark}
	The Gould-Hopper polynomials satisfy
	\begin{equation}\label{mf-G.H}
	H_{n}^{(p)}(az|c\gamma)=n!\sum_{k=0}^{\floor*{\frac{n}{p}}}\frac{(c-a^p)^{k}{\gamma}^k}{k!}\frac{a^{n-pk}}{(n-pk)!}H_{n-pk}^{(p)}(z|\gamma).
	\end{equation}
We note also that for  $p=2$, $q=0$, $w=1$, $c=1$ and $\gamma=-1$, we recover the multiplication formula for real Hermite polynomials $H_{n}$ in \cite[Eq. (4.6.33)]{mourad2005classical},
	while for $p=q=1$, $w=\bar{z}$,  $\gamma=-1$, $b=a$   and $c=1$, we find the multiplication formula for polyanalytic polynomials $H_{n,m}$ proved in \cite[Eq. (4.13)]{Ismail13b}.
\end{remark}

\section{Recurrence formulas}
In this section, we derive recurrence relations for the polynomials $H_{n,m}^{(p,q)}$. To this end, we begin by giving the action of the derivative operators with respect to $z$, $w$ and $\gamma$.
\begin{proposition} \label{pro:dp} The partial derivatives of $H_{n,m}^{(p,q)}(z,w|\gamma)$ are given by
	\begin{align}
	& \partial_{z}H_{n,m}^{(p,q)}(z,w|\gamma)=nH_{(n-1,m)}^{(p,q)}(z,w|\gamma) . \label{dpz}
	\\&\partial_{w}H_{n,m}^{(p,q)}(z,w|\gamma)=mH_{(n,m-1)}^{(p,q)}(z,w|\gamma). \label{dpw}
	\\&\partial_{\gamma}H_{n,m}^{(p,q)}(z,w|\gamma)=\partial_{z}^{p}\partial_{w}^{q}H_{n,m}^{(p,q)}(z,w|\gamma). \label{dpgamma}
	\end{align}
\end{proposition}

\begin{proof}
 We use the operational formula \eqref{opFor} combined with the fact that $e^{\gamma\partial_{z}^{p}\partial_{w}^{q}}$ and $\partial_{z}$ commute to get \eqref{dpz}. Indeed,
	$\partial_{z}H_{n,m}^{(p,q)}(z,w|\gamma)=e^{\gamma\partial_{z}^{p}\partial_{w}^{q}}\partial_{z}\{ z^{n}w^{m} \}=ne^{\gamma\partial_{z}^{p}\partial_{w}^{q}}\{z^{n-1}w^{m}\}.$
We obtain  \eqref{dpw} by the symmetry \ref{sym}. While the third partial derivative \eqref{dpgamma} is obtained by noticing that
$\partial_{\gamma}e^{\gamma\partial_{z}^{p}\partial_{w}^{q}}=\partial_{z}^{p}\partial_{w}^{q}e^{\gamma\partial_{z}^{p}\partial_{w}^{q}}$.
\end{proof}

\begin{remark}
	By mathematical induction, one can establish the following formula
	\begin{align}\label{D_z_w{j,k}}
\partial_{z}^{j}\partial_{w}^{k}H_{n,m}^{(p,q)}(z,w|\gamma)=
\frac{n!}{(n-j)!}\frac{m!}{(m-k)!}H_{n-j,m-k}^{(p,q)}(z,w|\gamma)
\end{align}
when $j\leq n$ and $k \leq m$. The left hand-side in \eqref{D_z_w{j,k}} vanishes otherwise.
Accordingly, we deduce
\begin{align}\label{D_gamma{k}}
\partial_{\gamma}^{k}H_{n,m}^{(p,q)}(z,w|\gamma)=\frac{n!}{(n-pk)!}\frac{m!}{(m-qk)!}H_{n-pk,m-kq}^{(p,q)}(z,w|\gamma)
\end{align}
if $k\leq \floor*{\frac{n}{p}}\wedge\floor*{\frac{m}{q}}$, and vanishes otherwise.
\end{remark}

Thanks to the previous proposition, we can assert the following.

\begin{proposition}
We have
	\begin{eqnarray}\label{zwH}
	z^{n}w^{m}= n!m!\sum_{k=0}^{\floor*{\frac{n}{p}}\wedge\floor*{\frac{m}{q}}}\frac{(-\gamma)^k}{k!}\frac{H_{n-pk,m-qk}^{(p,q)}(z,w|\gamma)}{(n-pk)!(m-qk)!}.
	\end{eqnarray}
		Subsequently, the following operational formula
	\begin{equation}\label{of}
	z^{n}w^{m}=e^{-\gamma\partial_{z}^{p}\partial_{w}^{q}}\{	H_{n,m}^{(p,q)}(z,w|\gamma)\}
	\end{equation}
	holds true.
\end{proposition}

\begin{proof}
	By writing down the Taylor series of the polynomials $H_{n,m}^{(p,q)}(z,w|\gamma+h)$, seen as function in the third variable,
	and next using \eqref{D_gamma{k}}, we get
	$$H_{n,m}^{(p,q)}(z,w|\gamma+h)=\sum_{k=0}^{n}\frac{h^k}{k!}\partial_{\gamma}^{k}H_{n,m}^{(p,q)}(z,w|\gamma)=n!m!\sum_{k=0}^{\floor*{\frac{n}{p}}\wedge\floor*{\frac{m}{q}}}\frac{h^k}{k!}\frac{H_{n-pk,m-qk}^{(p,q)}(z,w|\gamma)}{(n-pk)!(m-qk)!}.$$
	Then, by taking $h=-\gamma$ and using the fact that $H_{n,m}^{(p,q)}(z,w|0)=z^{n}w^m$, we arrive at \eqref{zwH}.

The proof of \eqref{of} follows from operation calculus, starting from the right hand-side and making appeal of Proposition \ref{propabcMu} with $a=b=1$ and $c=0$. Indeed, we have
\begin{align*}
e^{-\gamma\partial_{z}^{p}\partial_{w}^{q}}\{H_{n,m}^{(p,q)}(z,w|\gamma)  \}
&=n!m!\sum_{k=0}^{\floor*{\frac{n}{p}}\wedge\floor*{\frac{m}{q}}}\frac{(-\gamma)^k}{k!}\frac{H_{n-pk,m-qk}^{(p,q)}(z,w|\gamma)}{(n-pk)!(m-qk)!}=z^{n}w^{m}.
	\end{align*}
\end{proof}

The first recursion  relation in this section is the following.

\begin{proposition}\label{pro:rcd}
	The polynomials $H_{n,m}^{(p,q)}$ obey the recursion relations
	\begin{align}
	 &H_{n+1,m}^{(p,q)} (z,w|\gamma)=zH_{n,m}^{(p,q)}(z,w|\gamma)+{\gamma}p!q!\binom{n}{p-1}\binom{m}{q}H_{n+1-p,m-q}^{(p,q)}(z,w|\gamma), \label{rc1}
	 \\
	&	 H_{n+1,m}^{(p,q)}(z,w|\gamma)= (z+p\gamma\partial_{z}^{p-1}\partial_{w}^{q})H_{n,m}^{(p,q)}(z,w|\gamma).
	\label{rd1}
	\end{align}
\end{proposition}

\begin{proof}
	Starting from the generating function in \eqref{genRfg3} and replacing there $n$ by $n+p-1$ and $m$ by $m+q$, we get
	\begin{align*}
	(u^{p-1}v^{q})e^{zu+wv+{\gamma}u^{p}v^{q}}
	&=\sum_{n=p-1}^\infty \sum_{m=q}^\infty H_{n+1-p,m-q}^{(p,q)}(z,w|\gamma)\frac{u^{n}v^{m}}{(n+1-p)!(m-q)!}
	\\&=\sum_{n=0}^\infty\sum_{m=0}^\infty  (p-1)!q!\binom{n}{p-1}\binom{m}{q}H_{n+1-p,m-q}^{(p,q)}(z,w|\gamma)\frac{u^{n}v^{m}}{n!m!}.
	\end{align*}
But, since $\partial_{u}e^{zu+wv+{\gamma}u^{p}v^{q}}=(z+p{\gamma}u^{p-1}v^{q})e^{zu+wv+{\gamma}u^{p}v^{q}},$ we obtain  		$$\partial_{u}e^{zu+wv+{\gamma}u^{p}v^{q}}=\sum_{n=0}^\infty\sum_{m=0}^\infty  \left( zH_{n,m}^{(p,q)}(z,w|\gamma)+{\gamma}p!q!\binom{n}{p-1}\binom{m}{q}H_{n+1-p,m-q}^{(p,q)}(z,w|\gamma)\right) \frac{u^{n}v^{m}}{n!m!}.$$
	On the other hand, because of $H_{n+1,m}^{(p,q)}(z,w|\gamma)
	=\partial_{u}^{n}\partial_{v}^{m}(\partial_{u}e^{zu+wv+{\gamma}u^{p}v^{q}})\vline_{u=v=0},$
	we conclude that  $$H_{n+1,m}^{(p,q)}(z,w|\gamma)=zH_{n,m}^{(p,q)}(z,w|\gamma)+{\gamma}p!q!\binom{n}{p-1}\binom{m}{q}H_{n+1-p,m-q}^{(p,q)}.$$
	This completes our check of \eqref{rc1}. For the proof of \eqref{rd1}, notice first that
	\begin{align*}
	\partial_{z}^{p-1}\partial_{w}^{q}H_{n,m}^{(p,q)}(z,w|\gamma)
	=(p-1)!q!\binom{n'}{p-1}\binom{m'}{q}H_{n'+1-p,m'-q}^{(p,q)}.
	\end{align*}
	Therefore,
	\begin{align*}
	H_{n+1,m}^{(p,q)}(z,w|\gamma)&=zH_{n,m}^{(p,q)}(z,w|\gamma)+{\gamma}p!q!\binom{n'}{p-1}\binom{m'}{q}H_{n'+1-p,m'-q}^{(p,q)}(z,w|\gamma)
	\\&=(z+p\gamma\partial_{z}^{p-1}\partial_{w}^{q})H_{n,m}^{(p,q)}(z,w|\gamma).
	\end{align*}
\end{proof}

\begin{remark} The recursion relations
	\begin{align} &H_{n,m+1}^{(p,q)}(z,w|\gamma)=wH_{n,m}^{(p,q)}(z,w|\gamma) +{\gamma}p!q!\binom{n}{p}\binom{m}{q-1}H_{n-p,m-1-q}^{(p,q)}(z,w|\gamma) ,\label{rc2}
	\\&	 H_{n,m+1}^{(p,q)}(z,w|\gamma)=
	(w+q\gamma\partial_{z}^{p}\partial_{w}^{q-1})H_{n,m}^{(p,q)}(z,w|\gamma) \label{rd2}
	\end{align}
	follow from the previous ones  by the symmetry identity \eqref{sym}.
\end{remark}

In virtue of the previous recursion formulas, the polynomials $H_{n,m}^{(p,q)}$ can  be rewritten, according to the values of $p$ and $q$, in terms of some creation operators with the monomials $z^n$ and $w^m$ as generators. More precisely, we assert the following.

\begin{proposition}
	For any  $n,m,p,q=0,1,2, \cdots $, we have
	\begin{align}\label{creatingOp}
	H_{n,m}^{(p,q)}(z,w|\gamma)=\left\{\begin{array}{llll}
	e^{\gamma}z^{n}w^{m} , & \mbox{if } p=q=0 ,\\
	(z+p\gamma\partial_{z}^{p-1})^{n}\{w^{m}\}, & \mbox{if } p\geq1 \text{ and } q=0, \\
	(z+p\gamma\partial_{z}^{p-1}\partial_{w}^{q})^{n}\{w^{m}\}
	, & \mbox{if } p\geq1 \text{ and } q\geq 1.
 \end{array}	\right.
	\end{align}
\end{proposition}

\begin{proof}
	The first identity in \eqref{creatingOp} is obvious keeping in mind the convention that $[\frac{j}{k}]=+\infty$ when $k=0$. While the second one, i.e.,  when  $q=0$ and $p\geq 1$, can be derived making use of $H_{n}^{(p)}(z|\gamma)=(z+p{\gamma}\partial_{z}^{p-1})^{n}\{1\}$ in \cite[Eq (6), p. 18]{dattoli2009hermite}.
	Indeed, we have
	$ 
	H_{n,m}^{(p,0)}(z,w|\gamma)=w^{m}H_{n}^{(p)}(z|\gamma)=(z+p{\gamma}\partial_{z}^{p-1})^{n}\{w^m\}.
	$ 
 	The last identity, corresponding to $p\geq 1 $ and $q\geq 1$, can be handled by induction on $n$. Indeed,  we have
$H_{n,m}^{(p,q)}(z,w|\gamma)=(z+p\gamma\partial_{z}^{p-1}\partial_{w}^{q})H_{n-1,m}^{(p,q)}(z,w|\gamma),$
and therefore,
$H_{n,m}^{(p,q)}(z,w|\gamma)=(z+p\gamma\partial_{z}^{p-1}\partial_{w}^{q})^{n}H_{0,m}^{(p,q)}(z,w|\gamma),$
where $H_{0,m}^{(p,q)}(z,w|\gamma)=w^m$. Then, we have
$H_{n,m}^{(p,q)}(z,w|\gamma)=(z+p\gamma\partial_{z}^{p-1}\partial_{w}^{q})^{n}\{w^m\}.$
\end{proof}

\begin{remark}
	The analogues of second and third recursion formulas, with respect to $z$ variable, read
	$H_{n,m}^{(p,q)}(z,w|\gamma) =	(w+q\gamma\partial_{w}^{q-1})^{m}\{z^{n}\}$ and  $H_{n,m}^{(p,q)}(z,w|\gamma)=(w+q\gamma\partial_{z}^{p}\partial_{w}^{q-1})^{m}\{z^{n}\}$, respectively, and follows by the use of symmetry identity.
\end{remark}

Accordingly, we see that
 \begin{align*}
H_{n,m}^{(p,q)}(z,w|\gamma)&=(z+p\gamma\partial_{z}^{p-1}\partial_{w}^{q})^{n}H_{0,m}^{(p,q)}(z,w|\gamma)
\\&=(z+p\gamma\partial_{z}^{p-1}\partial_{w}^{q})^{n}(w+q\gamma\partial_{z}^{p}\partial_{w}^{q-1})^{m} H_{0,0}^{(p,q)}(z,w|\gamma).
\end{align*}
This can be reformulated as follows thanks to $H_{0,0}^{(p,q)}(z,w|\gamma)=1$.

\begin{corollary}
	We have
\begin{equation} 	
	H_{n,m}^{(p,q)}(z,w|\gamma)=(z+p\gamma\partial_{z}^{p-1}\partial_{w}^{q})^{n}(w+q\gamma\partial_{z}^{p}\partial_{w}^{q-1})^{m}(1).
\end{equation}
\end{corollary}


The next assertion is a recursion relation with respect to parameters $p$ and $q$.

\begin{proposition}
	We have
	\begin{equation}\label{RecFgamma}
	H_{n,m}^{(p+1,q)}(z,w|\gamma)=n!m!\sum_{k=0}^{\floor*{\frac{n}{p}}\wedge\floor*{\frac{m}{q}}}\sum_{j=0}^{k}\binom{j}{k}{\gamma}^k(-1)^{k-j}\frac{ H_{n-j-pk,m-k}^{(p,q)}(z,w|\gamma)}{(n-j-pk)!(m-qk)!}.
	\end{equation}
\end{proposition}

\begin{proof} Making  appeal of  the generating function $R^{p,q}_\gamma(z,w|u,v)$ in \eqref{genRfg3}, we get
	\begin{align*}
	R^{p+1,q}_\gamma(z,w|u,v) =e^{uz+vw+{\gamma}u^{p+1}v^q}
	=e^{uz+vw+({\gamma}u)u^{p}v^q}
	= R^{p,q}_{u\gamma}(z,w|u,v).
	\end{align*}
Therefore, in view of \eqref{mf2}, we obtain
	\begin{align*}
R^{p+1,q}_\gamma(z,w|u,v)&=\sum_{n=0}^\infty\sum_{m=0}^\infty \sum_{k=0}^{\floor*{\frac{n}{p}}\wedge\floor*{\frac{m}{q}}}{\gamma}^{k}(u-1)^{k}H_{n-pk,m-qk}^{(p,q)}(z,w|\gamma)\frac{u^{n}v^m}{n!m!}
	\\&=\sum_{n=0}^\infty\sum_{m=0}^\infty \sum_{k=0}^{\floor*{\frac{n}{p}}\wedge\floor*{\frac{m}{q}}}\sum_{j=0}^{k}\binom{j}{k}{\gamma}^k(-1)^{k-j}\frac{H_{n-j-pk,m-k}^{(p,q)}(z,w|\gamma)}{(n-j-pk)!(m-qk)!}\frac{u^{n}v^m}{n!m!}.
	\end{align*}
The result \eqref{RecFgamma} readily follows by identification.
\end{proof}

The previous recursion relation can be shown to equivalent to the operational formula \eqref{OpFF} below.

\begin{proposition}
	We have
	\begin{align}
	&H_{n,m}^{(p+1,q)}(z,w|\gamma)=e^{\gamma(\partial_{z}-1)\partial_{z}^{p}\partial_{w}^{q}}H_{n,m}^{(p,q)}(z,w|\gamma), \label{OpFF}\\
	&H_{n,m}^{(p,q+1)}(z,w|\gamma)=e^{\gamma(\partial_{w}-1)\partial_{z}^{p}\partial_{w}^{q}}H_{n,m}^{(p,q)}(z,w|\gamma), \label{opffs}\\
	&H_{n,m}^{(p+1,q+1)}(z,w|\gamma)=e^{\gamma(\partial_{z}+\partial_{w}-2)\partial_{z}^{p}\partial_{w}^{q}}H_{n,m}^{(p,q)}(z,w|\gamma). \label{oppp}
	\end{align}
\end{proposition}

\begin{proof}
	Clearly \eqref{opffs}  follows from \eqref{OpFF} by symmetry, while \eqref{oppp} is immediate consequence of combining \eqref{OpFF}  and \eqref{opffs} keeping in mind that  $(\partial_{z}-1)\partial_{z}^{p}\partial_{w}^{q}$
	and
	$(\partial_{w}-1)\partial_{z}^{p}\partial_{w}^{q}$  are commuting. 	
	We need only to prove \eqref{OpFF}. Starting from the definition of $H_{n,m}^{(p+1,q)}(z,w|\gamma)$
	and using \eqref{D_z_w{j,k}}, we obtain
	\begin{align*}
	H_{n,m}^{(p+1,q)}(z,w|\gamma)&=\sum_{k=0}^{\floor*{\frac{n}{p}}\wedge\floor*{\frac{m}{q}}}{\gamma}^{k}\sum_{j=0}^{k}\binom{j}{k}\partial_{z}^{j}(-1)^{k-j}\partial_{z}^{pk}\partial_{w}^{qk}H_{n,m}^{(p,q)}(z,w|\gamma)
	\\&=\sum_{k=0}^{\floor*{\frac{n}{p}}\wedge\floor*{\frac{m}{q}}}{\gamma}^{k}(\partial_{z}-1)^{k}(\partial_{z}^{p}\partial_{w}^{q})^kH_{n,m}^{(p,q)}(z,w|\gamma)
	\\&=e^{\gamma(\partial_{z}-1)\partial_{z}^{p}\partial_{w}^{q}}H_{n,m}^{(p,q)}(z,w|\gamma).
	\end{align*}
	\end{proof}

\begin{remark}
	These recursion relations are new even restricting to the Gould--Hopper polynomials.
\end{remark}

\section{Nielsen identities}
In this section, we prove some summation formulas of Nielsen type for the polynomials $H_{n,m}^{(p,q)}$ which can be used to derive others addition formulas.

\begin{theorem}\label{theoremNielsen}
We have
	\begin{align}
&H_{n+n',m}^{(p,q)}(z,w|\gamma)=\sum_{i=0}^{n}\sum_{j=0}^{n'}\binom{n}{i}\binom{n'}{j}(z-z')^{i+j}H_{n+n'-i-j,m}^{(p,q)}(z',w|\gamma) , \label{Nielsen1}\\
	&H_{n,m+m'}^{(p,q)}(z,w|\gamma)=\sum_{k=0}^{m}\sum_{l=0}^{m'}\binom{m}{k}\binom{m'}{l}(w-w')^{k+l}H_{n,m+m'-k-l}^{(p,q)}(z,w'|\gamma).\label{Nielsen2}
	\end{align}
\end{theorem}

\begin{proof} We need to prove only the first identity.
Indeed, by \eqref{fg1}
we have
	\begin{align}
	H_{m}^{(q)}(w|(u+t)^p\gamma)e^{z(u+t)}
	&=\sum_{n=0}^{\infty} H_{n,m}^{(p,q)}(z,w|\gamma)\frac{(u+t)^{n}}{n!}
	\nonumber\\&=\sum_{n=0}^\infty\sum_{n'=0}^\infty  H_{n+n',m}^{(p,q)}(z,w|\gamma)\frac{u^{n}}{n!}\frac{t^{n'}}{n'!}.\label{Eq1}
	\end{align}
	Applying this fact twice for given $z$ and $z'$, we find
	\begin{align*}
	\sum_{n=0}^\infty\sum_{n'=0}^\infty  H_{n+n',m}^{(p,q)}(z,w|\gamma)\frac{u^{n}}{n!}\frac{t^{n'}}{n'!}
	& = e^{(z-z')(u+t)}\sum_{n=0}^\infty\sum_{n'=0}^\infty  H_{n+n',m}^{(p,q)}(z',w|\gamma)\frac{u^{n}}{n!}\frac{t^{n'}}{n'!}
	\\&=\sum_{k=0}^\infty  \frac{(z-z')^{k}(u+t)^k}{k!}
	\left( \sum_{n=0}^\infty\sum_{n'=0}^\infty  H_{n+n',m}^{(p,q)}(z',w|\gamma)\frac{u^{n}}{n!}\frac{t^{n'}}{n'!}.\right)
	\\&
	=\sum_{k=0}^\infty \sum_{j=0}^k(z-z')^{k+j}\frac{u^k}{k!}\frac{t^j}{j!}\sum_{n,n'=0}H_{n+n',m}^{(p,q)}(z',w|\gamma)\frac{u^{n}}{n!}\frac{t^{n'}}{n'!}.
	\end{align*}
The last equality follows using  \cite[Eq. (1), p. 100]{manocha1984treatise}. Now, the substitution of $n$ by $n-k$ and $n'$ by $n'-j$
 lead to
	\begin{align*}
	\sum_{n=0}^\infty\sum_{n'=0}^\infty  H_{n+n',m}^{(p,q)}(z,w|\gamma)\frac{u^{n}}{n!}\frac{t^{n'}}{n'!}
	=\sum_{n=0}^\infty\sum_{n'=0}^\infty  n!n'!\sum_{k=0}^{n}\sum_{j=0}^{m}\frac{(z-z')^{k+j}}{k!j!}\frac{H_{n+n'-k-j,m}^{(p,q)}(z',w|\gamma)}{(n-k)!(n'-j)!}\frac{u^{n}}{n!}\frac{t^{n'}}{n'!}.
	\end{align*}
	The result in  \eqref{Nielsen1} follows by identification.
\end{proof}

A generalization of Theorem \ref{theoremNielsen} is the following one which readily follows from \eqref{Nielsen1}, applied to $H_{n+n',m+m'}^{(p,q)}(z,w|\gamma)$,
and \eqref{Nielsen2}. Set
$$ \binom{n,n',m,m'}{i,j,k,l} := \binom{n}{i}\binom{n'}{j}\binom{m}{k}\binom{m'}{l}.$$

\begin{proposition}\label{propNeilsen}
	We have
	\begin{align}\label{Niel}
	H_{n+n',m+m'}^{(p,q)}(z,w|\gamma)&=	\sum_{i=0}^{n}\sum_{j=0}^{n'}\sum_{k=0}^{m}\sum_{l=0}^{m'}
	\binom{n,n',m,m'}{i,j,k,l} \frac{(z-z')^{i+j}}{(w-w')^{-k-l}}
	 H_{n+n'-i-j,m+m'-k-l}^{(p,q)}(z',w'|\gamma).
	\end{align}
\end{proposition}	


\begin{remark}
	Notice that for $m=q=0$ in Proposition \ref{propNeilsen}, we get the formula for the Gould-Hopper polynomials \cite{khan2011summation}.
\end{remark}


As immediate consequence, we obtain the following addition formula with respect to the variables $z$ and $w$.

	\begin{corollary}
	We have
		\begin{equation}\label{add3}
	H_{n,m}^{(p,q)}(z+z',w+w'|\gamma)=\sum_{i=0}^{n}\sum_{j=0}^{m}\binom{n}{i}\binom{m}{j}z^{i}w^{j}H_{n-i,m-j}^{(p,q)}(z',w'|\gamma).
	\end{equation}
	\end{corollary}
	
	\begin{proof}
	This readily follows from \eqref{Niel} by specifying $n'=m'=0$, and replacing $z$ by $z-z'$ and $w$ by $w-w'$.
	\end{proof}

	\begin{corollary}
	We have
	\begin{equation}
H_{n,m}^{(p,q)}(z,w|\gamma)=2^{n+m}\sum_{i=0}^{n}\sum_{j=0}^{m}\binom{n}{i}\binom{m}{j}z^{i}w^{j}H_{n-i,m-j}^{(p,q)}(z,w|2^{p+q-1}\gamma).
\end{equation}
\end{corollary}

\begin{proof}
	It suffices to replace $z$ and $z'$ by $\frac{z}{2}$, and $w$ and $w'$ by $\frac{w}{2}$ in \eqref{add3}, and next applying \eqref{homogeneity}.
\end{proof}

	\section{Connection to Gould-Hopper polynomials}
	
	The main aim here is to express the polynomials $H_{n,m}^{(p,q)}$ in terms of the Gould-Hopper polynomials $H_{n}^{(p)}$ and vice-versa. We begin by expressing $H_{n}^{(p)}$ in function of $H_{n,m}^{(p,q)}$.

	\begin{proposition}
		We have
			\begin{equation}\label{HpHpq}
		H_{n}^{(p)}(z|\gamma)=\sum_{k=0}^{n}\binom{n}{k}H_{n-k,k}^{(p-q,q)}(z-w,w|\gamma) .
		\end{equation}
	\end{proposition}

	\begin{proof}
		The expression of $H_{n}^{(p)}$ in terms of $H_{n,m}^{(p,q)}$ as given through \eqref{HpHpq} is in fact equivalent to the following
				\begin{equation}\label{HpHpq1}
		H_{n}^{(p+q)}(z+w|\gamma)=\sum_{k=0}^{n}\binom{n}{k}H_{n-k,k}^{(p,q)}(z,w|\gamma)
		\end{equation}
		which readily follows by identification process. Indeed, by taking $ v=u$ in the generating function  \eqref{genRfg3}
and  substituting there $n$ by $n-k$, we obtain
\begin{align*}
\sum_{n=0}^\infty  \sum_{k=0}^{n}H_{n-k,k}^{(p,q)}(z,w|\gamma)\frac{u^{n}}{(n-k)!k!}
&=e^{(z+w)u+{\gamma}u^{p+q}}
=\sum_{n=0}^\infty  H_{n}^{(p+q)}(z|\gamma)\frac{u^{n}}{n!} .
\end{align*}
The last equality follows making use of the generating function \eqref{GenFctGH} for Gould--Hopper polynomials.
	\end{proof}

	\begin{remark} The specification of  $p=q=1$, $\gamma=-1$ and $m=0$ in \eqref{HpHpq} (or \eqref{HpHpq1}) provides us with a new expression of the holomrphic Hermite polynomials $H_{n}(z)$ in therms of the polyanalytic Ito--Hermite polynomials,
		$$H_{n}(z)=\sum_{k=0}^{n}\binom{n}{k}H_{n-k,k}(2i\Im(z),\bz).$$
	\end{remark}

Conversely, we assert the following.

	\begin{proposition}
		We have
		\begin{equation}\label{thmHtoGH}
		H_{n,m}^{(p,q)}(z,w|\gamma)=n!m!\sum_{k=0}^{\floor*{\frac{n}{p}}}\sum_{j=0}^{\floor*{\frac{m}{q}}}\sum_{l=0}^{\floor*{\frac{n-pk}{p}}}\sum_{i=0}^{\floor*{\frac{m-qj}{q}}}\frac{{(-2)^{-l-i}(-\gamma)}^{k+j}}{l!i!(k-l)!(j-i)!}\frac{H_{n-p(l+k)}^{(p)}(z|\gamma)}{(n-p(l+k))!}\frac{H_{m-q(i+j)}^{(q)}(w|\gamma)}{(m-q(i+j))!}.
		\end{equation}
	\end{proposition}

	\begin{proof}
		Using the generating function \eqref{genRfg3} and the fact that $e^{zu+wv+{\gamma}u^{p}v^{q}}=e^{zu+\frac{{\gamma}}{2}u^{p}v^{q}}e^{wv+\frac{{\gamma}}{2}u^{p}v^{q}}$,
	as well as the identity \eqref{mf-G.H}, it follows 
		\begin{align*}
		&\sum_{n=0}^\infty\sum_{m=0}^\infty  H_{n,m}^{(p,q)}(z,w|\gamma)\frac{u^{n}v^{m}}{n!m!}
		\\&=\sum_{n=0}^\infty\sum_{m=0}^\infty  n!\sum_{k=0}^{\floor*{\frac{n}{p}}}\frac{(\frac{v^{q}}{2}-1)^{k}{\gamma}^k}{k!}\frac{H_{n-pk}^{(p)}(z|\gamma)}{(n-pk)!}\frac{u^n}{n!}
		\left( m!\sum_{j=0}^{\floor*{\frac{m}{q}}}\frac{(\frac{u^{p}}{2}-1)^{k}{\gamma}^j}{j!}\frac{H_{m-qj}^{(q)}(w|\gamma)}{(m-qj)!}\frac{v^m}{m!}\right)
	   \\&=\sum_{n=0}^\infty\sum_{m=0}^\infty  \sum_{k=0}^{\floor*{\frac{n}{p}}}\sum_{j=0}^{\floor*{\frac{m}{q}}}\sum_{l=0}^{k}\sum_{i=0}^{j}\binom{k}{l}\binom{j}{i}(-1)^{k+j}(-2)^{-l-i}\frac{{\gamma}^{k+j}}{k!j!}\frac{H_{n-pk}^{(p)}(z|\gamma)}{(n-pk)!}\frac{H_{m-qk}^{(q)}(w|\gamma)}{(m-qk)!}u^{n+lp}v^{m+iq}
		\end{align*}
		Replacing $n+ip$ by $n$ and $m+iq$ by $m$, we get
	  \begin{align*}
	  \sum_{n=0}^\infty\sum_{m=0}^\infty  H_{n,m}^{(p,q)}(z,w|\gamma)\frac{u^{n}v^{m}}{n!m!}=&\sum_{n=0}^\infty\sum_{m=0}^{\infty}  n!m!\sum_{k=0}^{\floor*{\frac{n}{p}}}\sum_{j=0}^{\floor*{\frac{m}{q}}}\sum_{l=0}^{\floor*{\frac{n-pk}{p}}}\sum_{i=0}^{\floor*{\frac{m-qj}{q}}}\frac{{(-2)^{-l-i}(-\gamma)}^{k+j}}{l!i!(k-l)!(j-i)!}\\&\frac{H_{n-p(l+k)}^{(p)}(z|\gamma)}{(n-p(l+k))!}
	  \frac{H_{m-q(i+j)}^{(q)}(w|\gamma)}{(m-q(i+j))!}\frac{u^{n}v^{m}}{n!m!} .
	  \end{align*}
	  This completes the proof of \eqref{thmHtoGH}.
	\end{proof}

\section{Concluding remarks}
\label{SecHPDE}
The extension of the classical results valid for Hermite, Ito-Hermite and Gould--Hopper polynomials for the $(p,q)$ Gould--Hopper polynomials has been discussed and presented in an unified way.
It seems that the considered polynomials will be a fundamental tool in studying the high order partial differential equation in \eqref{hoheateq2}.
 In fact, the obtained recurrence formulas, in the previous section, show that the polynomials $H_{n,m}^{(p,q)}$ satisfy certain partial differential equations, generalizing, somehow, those obtained for the real and complex Hermite type polynomials considered in this paper. Indeed, from \eqref{dpgamma} we have
$$(\partial_{\gamma}-\partial_{z}^{p}\partial_{w}^{q})H_{n,m}^{(p,q)}(z,w|\gamma)=0.$$
This mean that the polynomials $H_{n,m}^{(p,q)}$ are solutions for the heat $(p,q)$-differential equation
$\partial_{\gamma}u =\partial_{z}^{p}\partial_{w}^{q}u.
$ 
Another partial differential equation satisfied by $H_{n,m}^{(p,q)}$ follows making use of \eqref{dpz} and the recursion relation \eqref{rc1}. Namely, we have
 \begin{align*}
&(z+\gamma\partial_{z}^{p-1}\partial_{w}^{q})\partial_{z}H_{n,m}^{(p,q)}(z,w|\gamma)
=nH_{n,m}^{(p,q)}(z,w|\gamma),
\\&
(w+\gamma\partial_{z}^{p}\partial_{w}^{q-1})\partial_{w}H_{n,m}^{(p,q)}(z,w|\gamma)=mH_{n,m}^{(p,q)}(z,w|\gamma).
\end{align*}
This means that the polynomials $H_{n,m}^{(p,q)}$ are eigenfunctions of the partial differential operators
$
\Delta^{p,q}_\gamma := 
z\partial_{z} +\gamma\partial_{z}^{p}\partial_{w}^{q} $ and $ \widetilde{\Delta^{p,q}_\gamma} := w\partial_{w} +\gamma\partial_{z}^{p}\partial_{w}^{q}$. Therefore they are solutions of the following system
\begin{equation}
\left  \{
\begin{array}{ll}
\Delta^{p,q}_\gamma u=nu,\\
\widetilde{\Delta^{p,q}_\gamma}u=mu.
\end{array}
\right.
\end{equation}
Subsequently and since the operator $\partial_z$ and $(w+\gamma\partial_{z}^{p}\partial_{w}^{q-1})$ commute, the polynomials $H_{n,m}^{(p,q)}$ obey to the following high order partial differential equation
\begin{equation}
(z+\gamma\partial_{z}^{p-1}\partial_{w}^{q})(w+\gamma\partial_{z}^{p}\partial_{w}^{q-1})\partial_{z}\partial_{w}u=nmu.
\end{equation}
The last equation means also that the polynomials $H_{n,m}^{(p,q)}$ are eigenfunction of the operator $(z+\gamma\partial_{z}^{p-1}\partial_{w}^{q})(w+\gamma\partial_{z}^{p}\partial_{w}^{q-1})\partial_{z}\partial_{w},$
corresponding to the eigenvalue $nm$.



\begin{thebibliography}{99}

	
	
	\bibitem{Barrett1984}   Barrett M.J., 
{ Nonlinear analysis of travelling wave tube amplifiers using complex Hermite polynomials}, Preprint, 1990.

\bibitem{BenGS2019}  Benahmadi A., Ghanmi A., Souid El Ainin M.,
{  A special orthogonal complement basis for holomorphic-Hermite functions and associated $1$d- and $2$d-fractional Fourier transforms}. Integral Transforms Spec. Funct.(2019).



\bibitem{chang2011gould} Chan J-H.,
{ The Gould-Hopper polynomials in the Novikov-Veselov equation},
Journal of Mathematical Physics, 2011, vol. 52, no 9, p. 092703.


\bibitem{dattoli1994theory} Dattoli G.,  Chiccoli C., Lorenzutta S., Maino G., Torre ,
{ A Theory of generalized Hermite polynomials}.
Computers and Mathematics with Applications, 1994, vol. 28, no 4, p. 71-83.

\bibitem{dattoli2009hermite} Dattoli, G and Germano, B and Ricci, PE,
{ Hermite polynomials with more than two variables and associated bi-orthogonal functions},
Integral Transforms and Special Functions, 2009, vol. 20, no 1, p. 17-22.

\bibitem{DunklXu14} Dunkl C., Xu Y.
{ Orthogonal polynomials of several variables.}
Encyclopedia of Mathematics and its Applications, 155.
2nd ed. Cambridge: Cambridge University Press; 2014.

\bibitem{Gh08jmaa} Ghanmi A.,
{ A class of generalized complex Hermite polynomials}.
J. Math. Anal. Appl. 340 ,2008, no. 2, 1395--1406.

\bibitem{Gh13ITSF}   Ghanmi A.,
{ Operational formulae for the complex Hermite polynomials $H_{p,q}(z, \bar z)$}.
{Integral Transforms Spec. Funct.},  Volume 24, Issue 11 (2013) pp  884-895.  

\bibitem{Gh17MMAS}  Ghanmi A.,
{ Mehler's formulas for the univariate complex Hermite polynomials and applications}. Math. Methods Appl. Sci. 40 ,2017, no. 18, 7540-7545.


\bibitem{Gorska2019holomorphic} G{\'o}rska, K., Horzela A., Szafraniec F.,H.,
{ Holomorphic Hermite polynomials in two variables}
Journal of Mathematical Analysis and Applications, 2019, vol. 470, no 2, p. 750-769

\bibitem{GouldHopper}  Gould H.W.,  Hopper A.T.,
{ Operational formulas connected with two generalizations of Hermite polynomials}. Duke Math. J.,1962, vol 29, p. 51-63.

\bibitem{Z-GLiu2017}  Liu Z-G., 
On the complex Hermite polynomials and partial differential equations. arXiv:1707.08708v2, 2017.

\bibitem{haimo1992representation} Haimo D.T., Markett C.,
{ A representation theory for solutions of a higher order heat equation, I}
Journal of mathematical analysis and applications, 1992, vol. 168, no 1, p. 89-107.

\bibitem{Hermite1864-1908} Hermite C.,
{ Sur un nouveau développement en série des fonctions},    Oeuvres de Charles Hermite, 1908, 293--308.

\bibitem{intissar2006spectral} Intissar, Abdelkader and Intissar, Ahmed;
{ Spectral properties of the Cauchy transform on $L2 (C, e-| z| 2\lambda (z))$},Journal of mathematical analysis and applications, 2006, vol. 313, no 2, p. 400-418.

\bibitem{mourad2005classical} Ismail M.E.H.,
{ Classical and quantum orthogonal polynomials in one variable},
Encyclopedia of Mathematics and its Applications,
98, 2005. 

\bibitem{Ismail13b}  Ismail M.E.H.,
{ Analytic properties of complex Hermite polynomials}.
Trans. Amer. Math. Soc. 368 , 2016, no. 2,p. 1189-1210.

\bibitem{ismail2015complex} Ismail M.E.H., Simeonov P.,
{ Complex Hermite polynomials: their combinatorics and integral operators},
Proc. Amer. Math. Soc. 143, 2015, no. 4,p. 1397-1410.

\bibitem{Ito52}  It\^o K.,
{ Complex multiple Wiener integral}.
Japanese journal of mathematics, 1952, vol. 22,p. 63-86.




\bibitem{khan2011summation} Khan S., Al-Saad, M.W.,
{ Summation formulae for Gould--Hopper generalized Hermite polynomials},
Computers and Mathematics with Applications, 2011, vol. 61, no 6, p. 1536-1541.

\bibitem{leveque2011sum}
L{\'e}v{\^e}que O., Vignat C., { About sum rules for Gould-Hopper polynomials},
arXiv preprint arXiv:1103.5168, 2011.

\bibitem{manocha1984treatise} Manocha H.L., Srivastava H.,
{ A treatise on generating functions},
Mathematics and Its Applications, Ellis Horwood Series, 1984.

\bibitem{nielsen1918recherches} Nielsen N.,
{ Recherches sur les polyn\^omes d'Hermite}.
Mathematisk-fysiske meddelelser. 79 pages,
Det Kgl. Danske Videnskabernes Selskab, 1918, vol. 1.

\bibitem{Szego75}  Szeg\"o G.,
{ Orthogonal polynomials. Fourth edition},
American Mathematical Society, Providence, R.I., 1975.


\bibitem{paris2010asymptotics} Paris, Richard B.,
{  Asymptotics of integrals of Hermite polynomials}.
Appl. Math. Sci, 2010, vol. 4, p. 3043-3056.

\bibitem{rainville1971special} Rainville, Earl D, {Special functions}, Chelsea Publishing Co., Bronx, NY, 1971.

\bibitem{Petojevic2018}  Petojevi\'c A.,
{ A Note about the Pochhammer Symbol},
Mathematica Moravica, 2008, (12-1), p. 37-42.

\bibitem{runge1914besondere}  Runge C.,
{ \"Uber eine besondere Art von Integralgleichungen},
Mathematische Annalen, 1914, vol. 75, no 1, p. 130-132.



\bibitem{TaimanovTsarev2008} Taimanov I.A.,  Tsarev S.P.,
Two-dimensional rational solitons and their blow-up via the Moutard transformation,
Theoret. Math. Phys. 157 (2) (2008) 1525--1541.


\bibitem{van1990new} van Eijndhoven S.J.L., Meyers J.L.H.,
{ New orthogonality relations for the Hermite polynomials and related Hilbert spaces},
Journal of Mathematical Analysis and Applications, 1990, vol 146, no. 1, p. 89-98.
\end{thebibliography}
\end{document}